\documentclass[12pt,a4paper,makeidx,reqno]{amsart}

\makeindex \topmargin=0in \headheight=8pt \headsep=0.5in \textheight=8.9in
\usepackage{amsmath,amssymb,amsthm}
\usepackage{mathrsfs}
\usepackage{eucal}
\usepackage[matrix,arrow,curve]{xy}
\usepackage{textcomp}
\usepackage{graphicx}
\usepackage{wrapfig}
\usepackage{verbatim}

\newtheorem{thma}{Theorem}
\newtheorem{lemma}{Lemma}[section]

\theoremstyle{remark}

\numberwithin{equation}{section}

\def\L{\CMcal{L}}

\def\div{\mathrm{div}}
\def\O{\CMcal{O}}
\def\supp{\mathrm{supp}}

\def\Disc{\mathrm{Discr}}
\def\hg{\widehat{g}}
\def\hpi{\widehat{\pi}}
\def\hB{\widehat{B}}
\def\Pic{\mathrm{Pic}}
\def\hPhi{\widehat{\Phi}}

\def\cM{\overline{\CMcal{M}}_g}
\def\cMpt{\overline{\CMcal{M}}_{g,1}}
\def\chM{\overline{\CMcal{M}}_{\hg}}

\def\MSigma{\mathfrak{M}_{g,\Sigma}^{(n)}}
\def\cMhit{\overline{\mathfrak{M}}_g^{(n)}}

\def\cQuad{\overline{\mathrm{\mathbf{Q}}}_g^N}

\def\cC{\CMcal{C}\cMhit}
\def\hC{\widehat{\CMcal{C}}\cMhit}

\def\hSigma{\widehat{\Sigma}}
\def\Vn{V_{\mathrm{nodal}}}

\def\cDW{\overline{D}_W}
\def\cnod{\overline{D}_W^{(b)}}
\def\cMax{\overline{D}_W^{(m)}}
\def\ccau{\overline{D}_W^{(c)}}

\def\hDW{\widehat{D}_W}
\def\hnod{\widehat{D}_W^{(b)}}
\def\hMax{\widehat{D}_W^{(m)}}
\def\hcau{\widehat{D}_W^{(c)}}

\def\B{\mathfrak{B}}
\def\hB{\widehat{\mathfrak{B}}}
\def\hlambda{\widehat{\lambda}}

\def\D{\CMcal{D}}
\def\hD{\widehat{\CMcal{D}}}
\def\Dmax{\CMcal{D}^{(m)}}
\def\Dcau{\CMcal{D}^{(c)}}

\def\Ddeg{D_{\mathrm deg}}

\begin{document}

\title{Discriminant and Hodge classes on the space of Hitchin's covers.}
\author{Mikhail Basok$^1$ }

\thanks{\textsc{${}^\mathrm{A}$ Laboratory of Modern Algebra and Applications, Department of Mathematics and
Mechanics, Saint-Petersburg State University, 14th Line, 29b, 199178 Saint-Petersburg, Russia.}}
\date{}

\bigskip

\begin{abstract}
  We continue the study of the rational Picard group of the moduli space of Hitchin's spectral covers started in~\cite{KorotkinZograf}. In the first part of the paper we expand the ``boundary'', ``Maxwell stratum'' and ``caustic'' divisors introduced in~\cite{KorotkinZograf} via the set of standard generators of the rational Picard group. This generalizes the result of~\cite{KorotkinZograf}, where the expansion of the full discriminant divisor (which is a linear combination of the classes mentioned above) was obtained. In the second part of the paper we derive a formula that relates two Hodge classes  in the rational Picard group of the moduli space of Hitchin's spectral covers. 
\end{abstract}

\maketitle

\section{Introduction}
\label{sec:introduction}

Hitchin's integrable systems arise as a result of dimensional reduction of the self-dual Yang-Mils equation, see~\cite{Hitchin1},~\cite{hitchin1987},~\cite{atiyah_1990}. The Hamiltonians of a Hitchin's system are encoded in the so-called \emph{spectral cover} $\hSigma$ (see~\cite{Donagi1995},~\cite{Donagi1996}) which is an $n$-sheeted cover of a (smooth or, more generally, stable) complex projective curve $\Sigma$ defined as a subvariety of $T^*\Sigma$:
\begin{equation}
  \hSigma = \{ (x,v)\in T^*\Sigma\ \mid\ P(v,x) = 0 \},
  \label{eq:def_of_hSigma}
\end{equation}
where
\begin{equation}
  P(v,x) = v^n + q_1(x) v^{n-1} + \dots + q_n(x),
  \label{eq:def_of_P(v,x)}
\end{equation}
$q_j$ is a $j$-differential on $\Sigma$ (i.e. a holomorphic section of $K_{\Sigma}^{\otimes j}$). In the framework of~\cite{Donagi1995}, the equation defining $\hSigma$ is given by the characteristic polynomial $P(v,x) = \det(\Phi(x) - vI)$ of the so-called Higgs field $\Phi$ on $\Sigma$.

We consider the moduli space $P\cMhit$ of Hitchin's spectral covers in the case of $\mathrm{GL}(n,\mathbb C)$ Hitchin's systems, when all differentials $q_j$ as assumed to be arbitrary. A point in $P\cMhit$ parametrizes a pair $(\Sigma, [P])$, where $\Sigma$ is a genus $g$ curve and $P$ is a polynomial of the form~\eqref{eq:def_of_P(v,x)} considered up to multiplication by a non-zero constant $\xi$ given by $(\xi\cdot P)(v,x) = \xi^nP(\xi^{-1}v, x)$. As a space $P\cMhit$ is a bundle over the Deligne-Mumford compactification of the moduli space $\cM$ of genus $g$ curves with fibers isomorphic to a weighted projective space, see Section~\ref{sec:space_of_covers} for details. Notice that if $n=1$ then $P\cMhit$ is just the total space of the projectivized Hodge bundle which can be thought of as a closure of the moduli space of Abelian differentials (considered up to a constant) on genus $g$ smooth projective curves. A. Kokotov and D. Korotkin~\cite{kokotov2009} introduced a tau function on this moduli space called Bergman tau function. The Bergman tau function is a generalization of the Dedeking eta function (they coincide if $g=1$) and can be interpreted as determinants of a family of Cauchy-Riemann operators in the spirit of~\cite{Palmer1990}. Studying the asymptotycs of the Bergman tau function near the boundary of the moduli space of Abelian differentials (embedded into the total space of the Hodge bundle) D. Korotkin and P. Zograf~\cite{KZAbDif} developed a new relation in the rational Picard group of the projectivized Hodge bundle. Thereafter the construction of the Bergrman tau function was generalized to the case of the moduli space of n-differentials (which is closely related with $P\cMhit$ when $n=2$, see~\cite{KZprym} for $n=2$ and~\cite{KorotkonSauvagetZograf} for $n>2$), and finally to the case of $P\cMhit$ for any $n$~\cite{KorotkinZograf}. Study of the properties of the Bergman tau function on $P\cMhit$ allows to express the class of the full Discriminant locus in the rational Picard group of $P\cMhit$ via the set of its standard generators (see Theorem~\ref{thma:Korotkin_Zograf_theorem}). The first objective of the current paper is to enhance and specify this result using standard methods of algebraic geometry.

Let $\Sigma$ be a smooth genus $g$ curve and let $P$ be a polynomial of the form~\eqref{eq:def_of_P(v,x)}. Then the discriminant $W(x) = \Disc(P(\cdot, x))$ is an $n(n-1)$-differential on $\Sigma$ and the divisor of $W$ is equal to the branching divisor of the spectral cover $\hSigma\to \Sigma$ associated with $P$. Generically, all zeros of $W$ are simple which implies that $\hSigma$ is smooth and the cover $\hSigma\to \Sigma$ is simply ramified. When two zeros of $W$ coalesce, the local behavior of the map $\hSigma\to \Sigma$ changes with respect to one the following three ways (we follow the notation of~\cite{KorotkinZograf} in this description):

1) A node (normal self-crossing of $\hSigma$) occurs at the ramification point of $\hSigma\to \Sigma$ over the double zero of $W$. We call the locus of such covers \textbf{The ``boundary''.}.

2) Two distinct ramification points of $\hSigma\to\Sigma$ arise in the preimage of the double zero of $W$. We call the locus of such covers \textbf{The ``Maxwell stratum''.} 

3) Two ramification points of $\hSigma\to \Sigma$ coalesce to become a ramification point of order $3$ over the double zero of $W$. We call the locus of such covers  \textbf{The ``caustic''.}.

We use the notation of a "Maxwell stratum" and a "caustic" following the notation developed in V. Arnold's school in Moscow, see for example~\cite{lando2010graphs}.

The correspondence $P\mapsto \Disc(P)$ defines a map from $P\cMhit$ to the moduli space of pairs $(\Sigma, W)$, where $W$ is an $n(n-1)$-differential on $\Sigma$ considered up to a multiplication by a non-zero constant. Let $P\cDW$ denote the pullback of the divisor consisting of those $W$ that have at least one multiple zero (see Section~\ref{sec:discriminant_locus} for details). Then the support of $P\cDW$ splits into the union of three components $P\cnod\cup P\cMax\cup P\ccau$ in accordance to the three possibilities described above. We call the divisor $P\cDW$ \emph{the full discriminant divisor}. The class of the divisor $P\cDW$ in the rational Picard group of $P\cMhit$ is called \emph{the class of the universal Hitchin's discriminant}. The following theorem was proven in~\cite[Theorem~3.2]{KorotkinZograf}:

\begin{thma}
  The divisor $P\cDW$ satisfies
  \begin{equation*}
    P\cDW =P\cnod + 2P\cMax + 3P\ccau
  \end{equation*}
  and the class of the universal Hitchin's discriminant $P\cDW$ is expressed in terms of the standard generators of $\Pic(P\cMhit)\otimes \mathbb Q$ as follows:
  \begin{equation*}
    [P\cDW] = n(n-1)\Bigl( (n^2-n+1)(12\lambda-\delta) - 2(g-1)(2n^2-2n+1)\phi \Bigr).
  \end{equation*}
  \label{thma:Korotkin_Zograf_theorem}
\end{thma}

Here $\delta = \sum_{j=0}^{[g/2]}\delta_j$ is the pullback of the class of the Deligne-Mumford boundary of $\cM$, see Section~\ref{subsec:generators_of_Pic} for details. We generalize this result by expressing the class of each of the three components of the full discriminant divisor via the set of generators of $\Pic(P\cMhit)\otimes\mathbb Q$:

\begin{thma}
  Let $n\geq 3$ and $g\geq 1$. The following formulas hold in $\Pic(P\cMhit)\otimes \mathbb Q$:
  \begin{align*}
    & [P\cnod] = n(n-1)\Bigl( (n+1)(12\lambda-\delta) - 2(g-1)(2n+1)\phi \Bigr)\\
    & [P\cMax] = \frac{n(n-1)(n-2)(n-3)}{2}\Bigl( 12\lambda-\delta + 4(g-1)\phi \Bigr)\\
    & [P\ccau] = n(n-1)(n-2)\Bigl( 12\lambda-\delta - 4(g-1)\phi \Bigr).
  \end{align*}
  \label{thma:formulas_for_three_divisors}
\end{thma}

The second objective of the paper is to relate two Hodge classes on $P\cMhit$. Note that since the degree of the cover $\hSigma\to\Sigma$ is equal to $n$ and the degree of the branching divisor is $\deg\div(W) = 2n(n-1)(g-1)$, the genus of $\hSigma$ is $\hg=g(\hSigma) = n^2(g-1)+1$. Thus we have two morphisms $P\cMhit\to \cM$ and $P\cMhit\to \chM$, where the first morphism maps $(\Sigma, [P])$ to the moduli of $\Sigma$ and the second one maps $(\Sigma, [P])$ to the moduli of $\hSigma$. Hence we can define two Hodge classes: the class $\lambda$ is the pullback of the Hodge class from $\cM$ and the class $\hlambda$ is the pullback of the Hodge class from $\chM$. The next theorem provides a formula in $\Pic(P\cMhit)\otimes \mathbb Q$ which relates $\lambda$ and $\hlambda$:

\begin{thma}
  Let $n\geq 3$ and $g\geq 1$. The following formula holds in $\Pic(P\cMhit)\otimes \mathbb Q$:
  \begin{equation*}
    \hlambda = n(2n^2-1)\lambda - \frac{n(n-1)(4n+1)(g-1)}{6}\phi - \frac{n(n^2-1)}{6}\delta.
  \end{equation*}
  \label{thma:hlambda_formula}
\end{thma}

Note that all coefficients of the right-hand side in the formula in Theorem~\ref{thma:hlambda_formula} are integers.

The paper is organized as follows. In Section~\ref{sec:polynomials_variety} we recall some basic facts about the geometry of the space of monic polynomials with multiple roots and derive a technical lemma that will be used for a local analysis of the discriminant $W$ with multiple zeros. In Section~\ref{sec:space_of_covers} we recall the construction of $P\cMhit$ and some of its basic properties, and introduce some related notation. In Section~\ref{sec:discriminant_locus} we prove Theorem~\ref{thma:formulas_for_three_divisors} and in Section~\ref{sec:hlambda} we prove Theorem~\ref{thma:hlambda_formula}.

\subsection*{Acknowledgments.} I would like to thank Dmitry Korotkin and Peter Zograf for helpful discussions. The research was supported by the grant of the Government of the Russian Federation for the state support of scientific research carried out under the supervision of leading scientists, agreement 14.W03.31.0030 dated 15.02.2018.

\section{Variety of monic polynomials}
\label{sec:polynomials_variety}

Recall that a polynomial of the form $P(t) = t^n + q_1t^{n-1}+\dots+q_n$ is called \emph{monic}. In this section we assume that $(q_1,\dots,q_n)\in \mathbb C^n$. The discriminant $\Disc(P)$ is a polynomial of $q_1,\dots,q_n$ and the equation $\Disc(P) = 0$ defines the affine variety $\D\subset \mathbb C^n$ of monic polynomials with multiple roots (see~\cite{lando2010graphs}). The variety $\D$ is not smooth: it has a normal crossing along the subvariety $\Dmax\subset \D$ that parametrizes polynomials with two multiple roots, and it has a cusp along the subvariety $\Dcau\subset \D$ corresponding to polynomials that have a root of order $3$ or bigger. To construct the normalization of $\D$ let us consider the variety $\hD\subset \mathbb C^n\times \mathbb C$ given by
\begin{equation}
  \hD = \{(q_1,\dots, q_n, t)\in  \mathbb C^n\times \mathbb C \mid\ P(t) = 0,\ P'(t) = 0\},
  \label{eq:def_of_hD}
\end{equation}
where $P(t) = t^n + q_1t^{n-1}+\dots+q_n$. The forgetful projection $ \mathbb C^n\times \mathbb C\to \mathbb C^n$ maps $\hD$ onto $\D$, and the induced mapping $\hD\to \D$ has the degree one over an open subset of $\D$. To see that $\hD$ is smooth observe that the functions $t, P(t),P'(t),\dots,P^{(n-1)}(t)$ form another coordinate system on $\mathbb C^n\times \mathbb C$. In this coordinates $\hD$ is just a linear subspace given by two linear equations $P(t) = 0,\ P'(t)=0$.

The group $\mathbb C^*$ acts on the space $\mathbb C^n$ of monic polynomials by the rule $(\xi\cdot P)(t) = \xi^n P(\xi^{-1}t)$. In terms of the coefficients $q_1,\dots, q_n$ this action can be rewritten as
\begin{equation}
  \xi\cdot (q_1,q_2, \dots, q_n) = (\xi q_1, \xi^2 q_2,\dots, \xi^n q_n).
  \label{eq:action_of_Cstar_on_q}
\end{equation}
Denote by $P\mathbb C^n$ the projectivization under this action. The weighted projective space $P\mathbb C^n$ is a smooth orbifold. The variety $\D$ is equivariant under the action of $\mathbb C^*$ and we denote its projectivization by $P\D$. The action of $\mathbb C^*$ on $\mathbb C^n$ lifts to the action of $\mathbb C^*$ on $\mathbb C^n\times \mathbb C$ given by $\xi\cdot(P,t) = (\xi\cdot P, \xi t)$. The variety $\hD$ is equivariant under this action and we denote the projectivization by $P\hD$. Note that the map $\hD\to \D$ induces the map $P\hD\to P\D$ (although we do not have a map from $P(\mathbb C^n\times \mathbb C)$ to $P\mathbb C^n$) of projectivized varieties.

\begin{lemma}
  Let $n\geq 3$. Given a monic polynomial $P(t) = t^n + q_1t^{n-1} + \dots + q_n$ we denote by $P_{n-2}(t) = t^{n-2} + q_1 t^{n-3} + \dots + q_{n-2}$ the polynomial $t^{-2}(P(t) - (q_{n-1}t - q_n))$. There exist polynomials $R_0,R_1,S\in \mathbb C[q_1,\dots, q_n]$ such that the equation
  \begin{equation*}
    \Disc(P) = - q_nq_{n-2}^3\Disc(P_{n-2}) + q_n(q_{n-1}R_1 + q_n R_0) + q_{n-1}^2 S
  \end{equation*}
  holds for any choice of the polynomial $P(t) = t^n + q_1t^{n-1} + \dots + q_n$.
  \label{lemma:decomposition_of_Disc}
\end{lemma}

\begin{proof}
  Clearly, $\Disc(P)$ lies in the ideal in $\mathbb C[q_1,\dots,q_n]$ generated by $q_n$ and $q_{n-1}$, so we can find $S_1,S_2\in \mathbb C[q_1,\dots, q_n]$ such that $\Disc(P) = q_n S_1 + q_{n-1}S_2$ where $S_2$ is independent of $q_n$. Let $S_2 = q_{n-1} S + S_3$, where $S_3$ is independent on $q_{n-1}$. We obtain
  \begin{equation}
    \Disc(P) = q_n S_1 + q_{n-1}^2S + q_{n-1}S_3.
    \label{eq:intermediate_decomp_of_Disc_-1}
  \end{equation}
  Consider the polynomial
  \begin{equation*}
    P_z(t) = (t-z)(t-2z)P_{n-2}(t) = t^n + q_1(z) t^{n-1} + \dots + q_n(z)
  \end{equation*}
   where $z$ is a formal variable. Then $q_n(z)$ is divisible by $z^2$ and $q_{n-1}(z)$ is divisible by $z$ but not by $z^2$. We have
  \[\begin{array}{lll}
      \Disc(P_z)&= &q_n(z) S_1(q_1(z),\dots,q_n(z)) + q_{n-1}(z)^2 S(q_1(z),\dots,q_{n-1}(z))+\\
      & &\hspace{10em}+ q_{n-1}(z) S_3(q_1(z),\dots,q_{n-2}(z))\\
      &= &q_{n-1}(z)S_3(q_1,\dots,q_{n-2}) + z^2Q_1 
  \end{array}\]
  due to~\eqref{eq:intermediate_decomp_of_Disc_-1}, where $Q_1\in \mathbb C[z, q_1,\dots,q_n]$ is some polynomial. Since $\Disc(P_z)$ is divisible by $z^2$, we see that $S_3$ must be zero, hence $S_2 = q_{n-1} S$ and 
  \begin{equation}
    \Disc(P) = q_n S_1 + q_{n-1}^2S.
    \label{eq:intermediate_decomp_of_Disc_0}
  \end{equation}

  We can find $R_0,R_1,R_2,R_3\in \mathbb C[q_1,\dots,q_n]$, where the polynomial $R_2$ is independent of $q_n, q_{n-1}$ and $R_3$ is independent of $q_n, q_{n-1}, q_{n-2}$, such that $S_1 = q_n R_0 + q_{n-1} R_1 + q_{n-2} R_2 + R_3$. Equation~\eqref{eq:intermediate_decomp_of_Disc_0} then looks like follows
  \begin{equation}
    \Disc(P) = q_nq_{n-2}R_2(q_1,\dots,q_{n-2}) + q_n(q_{n-1}R_1 + q_n R_0) + q_{n-1}^2S + q_nR_3(q_1,\dots,q_{n-3}).
    \label{eq:intermediate_decomp_of_Disc_1}
  \end{equation}
  Put $P_{n-3}(t) = t^{n-3} + q_1 t^{n-2} + \dots + q_{n-3}$, and consider the polynomial
  \begin{equation*}
    \begin{split}
      & P_z(t) = (t-3z)(t-6z)(t+2z)P_{n-3}(t) =\\
      &= (t^3-7zt^2 +36z^3)P_{n-3}(t) = t^n + q_1(z) t^{n-1} + \dots + q_n(z)
    \end{split} 
  \end{equation*}
   where $z$ is a formal variable. Then $q_n(z)$ and $q_{n-1}(z)$ are divisible by $z^3$, and $q_{n-2}(z)$ is divisible by $z$. Using these observations together with~\eqref{eq:intermediate_decomp_of_Disc_1} we conclude that
  \begin{equation*}
    \Disc(P_z) = q_n(z)R_3(q_1,\dots, q_{n-3}) + z^4Q_2,
  \end{equation*}
  where $Q_2\in \mathbb C[z, q_1,\dots,q_n]$ is some polynomial. Note that the $\Disc(P_z)$ is divisible by $z^4$ (it is even divisible by $z^6$). Since $q_n(z)$ is not divisible by $z^4$, the polynomial $R_3$ should be zero. It follows that we can now rewrite~\eqref{eq:intermediate_decomp_of_Disc_1} as
  \begin{equation}
    \Disc(P) = q_nq_{n-2}R_2(q_1,\dots,q_{n-2}) + q_n(q_{n-1}R_1 + q_n R_0) + q_{n-1}^2S.
    \label{eq:intermediate_decomp_of_Disc_2}
  \end{equation}
  Using this equation and the factorization properties of $q_n(z), q_{n-1}(z), q_{n-2}(z)$ again we conclude that
  \begin{equation*}
    \Disc(P_z) = q_n(z)q_{n-2}(z) R_2(q_1(z), \dots,q_{n-2}(z)) + z^6Q_3,
  \end{equation*}
  where $Q_3\in \mathbb C[z, q_1,\dots,q_n]$ is some polynomial. Since $\Disc(P_z)$ is divisible by $z^6$ while $q_n(z)q_{n-2}(z)$ is divisible only by $z^4$, the polynomial $R_2(q_1(z), \dots,q_{n-2}(z))$ must be divisible by $z^2$. It can be easily shown that this is possible only if $R_2(q_1,\dots, q_{n-2})$ is divisible by $q_{n-2}^2$, i.e. there exists a polynomial $R_4\in \mathbb C[q_1,\dots,q_{n-2}]$ such that
  \begin{equation}
    \Disc(P) = q_nq_{n-2}^3R_4(q_1,\dots,q_{n-2}) + q_n(q_{n-1}R_1 + q_n R_0) + q_{n-1}^2S.
    \label{eq:intermediate_decomp_of_Disc_3}
  \end{equation}

  Now consider the polynomial
  \begin{equation*}
    \begin{split}
      & P_z(t) = (t^2-z) P_{n-2}(t) =\\
      & = t^n + q_1 t^{n-1} + (q_2 - z)t^{n-2} + (q_3-zq_1)t^{n-3} + \dots + (q_{n-2} - zq_{n-4})t^2 -zq_{n-3}t - zq_{n-2}
    \end{split}
  \end{equation*}
  It follows from~\eqref{eq:intermediate_decomp_of_Disc_3} that
  \begin{equation}
    \Disc(P_z) = -zq_{n-2}^4 R_2(q_1,\dots, q_{n-2}) + z^2Q_4,
    \label{eq:asymptotics_of_Pz_Disc}
  \end{equation}
  where $Q_4\in \mathbb C[z, q_1,\dots,q_n]$ is some polynomial. On the other hand we have
  \[
    \Disc(P_z) = z\cdot\Disc(P_{n-2})\cdot \mathrm{Res}(t^2-z, P_{n-2}(t))^2 = zq_2^4\cdot\Disc(P_{n-2}) + z^2Q_5,
  \]
  where $Q_5\in \mathbb C[z, q_1,\dots,q_n]$ is some polynomial. Comparing this equation with~\eqref{eq:asymptotics_of_Pz_Disc} we find that $R_2 = -\Disc(P_{n-2})$. Substituting this equality into~\eqref{eq:intermediate_decomp_of_Disc_3} we get the statement of the lemma.
  \end{proof}

\section{Space of covers}
\label{sec:space_of_covers}

Let $\Sigma$ be a smooth projective curve of genus $g$. Denote by $\MSigma$ the moduli space of $\mathrm{GL}(n,\mathbb C)$ spectral covers of $\Sigma$:
\begin{equation}
  \MSigma = \bigoplus_{j = 1}^n H^0(\Sigma, K_{\Sigma}^{\otimes j})
  \label{eq:def_of_Mhit_Sigma}
\end{equation}
where $K_{\Sigma}$ is the canonical class of $\Sigma$ and
\begin{equation}
  \dim\MSigma = n^2(g-1)+1.
  \label{eq:def_of_hg}
\end{equation}
A point $(q_1,\dots,q_n)\in \MSigma$ can be considered as a polynomial $P(t,x) = t^n + q_1(x)t^{n-1} + \dots + q_n(x)$. For each $x\in \Sigma$ and $v\in T_x^*\Sigma$ the value $P(v,x)$ is an element of $(T^*_x\Sigma)^{\otimes n}$. The spectral cover $\hSigma$ associated with $P$ is a subvariety in $T^*\Sigma$ defined by
\begin{equation*}
  \hSigma = \{(x, v)\in \Sigma\times T_x^*\Sigma \mid\ P(v,x) = 0\};
\end{equation*}
clearly, $\hSigma$ is a projective curve. Generically $\hSigma$ is smooth and all the ramification points of the projection $\hSigma\to\Sigma$ are simple. If $\hSigma$ is smooth, then by the Riemann-Hurwitz formula the genus of $\hSigma$ is equal to $\hg = n^2(g-1)+1$. Define the action of $\mathbb C^*$ on $\MSigma$ by 
\begin{equation}
  (\xi\cdot P)(t,x) = \xi^n P(\xi^{-1}t, x).
  \label{eq:def_of_action_of_Cstar}
\end{equation}
Denote by $P\MSigma$ the corresponding projectivization.

Let $\cM$ be the Deligne-Mumford compactification of the moduli space of genus $g$ curves and let $\nu: \cMpt\to \cM$ be the universal curve. We define the moduli space of $\mathrm{GL}(n,\mathbb C)$ Hitchin's spectral covers by
\begin{equation}
  \cMhit = \bigoplus_{j = 1}^n R^0\nu_*\omega_{\nu}^{\otimes j},
  \label{eq:def_of_cMhit}
\end{equation}
where $\omega_{\nu}$ is the relative dualizing sheaf. The forgetful projection $\cMhit\to \cM$ is a bundle with fiber over $\Sigma$ isomorphic to $\MSigma$ (in the case when $\Sigma$ is not smooth one have to replace $K_{\Sigma}$ with the relative dualizing sheaf on $\Sigma$). The action of $\mathbb C^*$ on $\MSigma$ defined by~\eqref{eq:def_of_action_of_Cstar} extends to the action on $\cMhit$. Let $P\cMhit$ denote the corresponding projectivization. The $P\cMhit$ is a smooth orbifold (or a Deligne-Mumford stack). Denote by $\L\to P\cMhit$ the tautological line bundle associated with the projectivization.

Let $\pi:\cC\to P\cMhit$ be the pullback of the universal curve $\cMpt\to\cM$. Denote by $\omega_{\pi}$ the relative dualizing sheaf and set 
\begin{equation}
  \psi = c_1(\omega_{\pi}).
  \label{eq:def_of_psi_1}
\end{equation}
Let $\hpi: \hC\to P\cMhit$ be the universal family of Hitchin's spectral curves, so that the fiber of $\hpi$ over a point $(\Sigma, [P])\in P\cMhit$ is isomorphic to the curve $\hSigma$ associated with $P$. A point in $\hC$ can be represented by a quadruple $(\Sigma,x, P, v)$, where $P\in \MSigma$, $x\in \Sigma$, $v\in T_x^*\Sigma$ and $P(v,x) = 0$. It is straightforward to check that the map $p:\hC\to \cC$ that forgets $v$ is a branched cover that coincides with $\hSigma\to \Sigma$ fiber-wise over $P\cMhit$. Let us denote by $\hB\subset \hC$ the ramification divisor of $p$ and by $\B = p(\hB)\subset \cC$ the branching divisor. Consider the projection $\hB\to \cMpt$ that maps $(\Sigma, x, P, v)$ to $(\Sigma, x)$. Let $(\Sigma,x)$ be curve with a marked point $x$ and assume for simplicity that $\Sigma$ is smooth (otherwise one has to consider the normalization of $\Sigma$). Let $(\MSigma)_x\subset \MSigma$ be the subvariety consisting of $(q_1,\dots,q_n)$ such that $q_j(x) = 0$ for each $j$. The fiber of $\hB\to \cMpt$ is (not canonically) isomorphic to $P(\hD\times (\MSigma)_x)$. Similarly, the fiber of the projection $\B\to \cMpt$ is isomorphic to $P(\D\times (\MSigma)_x)$ (see Section~\ref{sec:polynomials_variety}, where we define $\D$ and $\hD$). Notice that
\begin{equation}
  (\MSigma)_x \simeq \mathbb C^{n^2(g-1)-n+2}
  \label{eq:Mhit_x_Sigma_isomorphic_to_C_ghat-n}
\end{equation}
(cf.~\eqref{eq:def_of_hg}). Define the action of $\mathbb C^*$ on $\mathbb C^{n^2(g-1)-n+2}$ via this isomorphism. The following lemma is straightforward:
\begin{lemma}
  (1) The projection $\hB\to \cMpt$ is a bundle with the fiber $P(\hD\times \mathbb C^{n^2(g-1)-n+2})$ (i.e. it can be locally represented as a projection of the form $P(\hD\times \mathbb C^{n^2(g-1)-n+2})\times X\to X$). In particular, the variety $\hB$ is smooth.
  
  (2) The projection $\B\to \cMpt$ is a bundle with fiber $P(\D\times \mathbb C^{n^2(g-1)-n+2})$. In particular, the singularities of $\B$ are normal crossings and cusps.
  
  (3) The map $\hB\to \B$ is a bundle morphism that is given fiber-wise by the map $\hD\to \D$.
  
  (4) The ramification of the map $p: \hC\to \cC$ is simple at a generic point.
  \label{lemma:hB_is_smooth}
\end{lemma}

\section{Components of the universal discriminant locus}
\label{sec:discriminant_locus}

Let $P = t^n + q_1t^{n-1} + \dots + q_n$ represent an element in $\MSigma$. Consider the discriminant $W(x) = \Disc(P(\cdot, x))$. Recall that $W$ is an $N$-differential, where $N = n(n-1)$, and the divisor of $W$ is equal to the branching divisor of the spectral cover $\hSigma\to \Sigma$ associated with $P$. Generically all zeros of $W$ are simple, thus $\hSigma$ is smooth and the ramification of $\hSigma\to \Sigma$ is simple. If $x$ is a zero of order $2$ of $W$ then there are three possibilities that describe the local behaviour of the cover $\hSigma\to\Sigma$; we will follow the notation of~\cite{KorotkinZograf}:

1)  There is one simple ramification point of $\hSigma\to \Sigma$ over $x$ and $\hSigma$ has a node (normal crossing) at this point. We call the locus of such covers \textbf{the ``boundary''}.

2) The cover $\hSigma\to \Sigma$ has two simple ramification points of order $2$ over $x$ and $\hSigma$ is smooth at these points. We call the locus of such covers \textbf{the ``Maxwell stratum''}. 

3) The cover $\hSigma\to \Sigma$ has a ramification point of order $3$ over $x$ and is $\hSigma$ is smooth at this point. We call the locus of such covers \textbf{the ``caustic''}.

Let $\cQuad$ be the moduli space of pairs $(\Sigma, W)$, where $\Sigma$ is a curve of genus $g$ and $W$ is an $N$-differential on it (or a section of $\omega_{\Sigma}^{\otimes N}$ if $\Sigma$ is not smooth). The map $P\mapsto \Disc(P)$ gives rise to a map $\Disc:\cMhit\to \cQuad$. Let $\Ddeg$ be the divisor in $\cQuad$ parametrizing pairs $(\Sigma, W)$, where $W$ has multiple zeros. The locus $\Disc^{-1}(\Ddeg)$ has three components $\cnod$, $\cMax$ and $\ccau$ in accordance with the three possibilities described above. Put $\cDW = \Disc^*\Ddeg$. A local analysis (see~\cite{KorotkinZograf}) yields $\cDW = \cnod + 2\cMax + 3\ccau$ (alternatively, one can use Lemma~\ref{lemma:hB_is_smooth} to show this). We call the divisor $\cDW$ the \emph{universal Hitchin's discriminant}. Note that $\cDW$, $\cnod$, $\cMax$ and $\ccau$ are equivariant under the action of $\mathbb C^*$ on $\cMhit$. Therefore, we can define their projectivizations $P\cDW$, $P\cnod$, $P\cMax$ and $P\ccau$ that are divisors in $P\cMhit$. Our goal is to represent the classes of $P\cnod$, $P\cMax$ and $P\ccau$ as linear combinations of standard generators of $\Pic(P\cMhit)\otimes \mathbb Q$.

\subsection{Generators of $\Pic(\cMhit)\otimes \mathbb Q$}
\label{subsec:generators_of_Pic}

Put 
\begin{equation}
  \phi= c_1(\L),
  \label{eq:def_of_phi}
\end{equation}
where $\L\to P\cMhit$ is the tautological line bundle associated with the action of $\mathbb C^*$ on $\cMhit$. By construction, the space $P\cMhit$ is a bundle over $\cM$ whose fibers are weighted projective spaces. Therefore $\Pic(P\cMhit)\otimes \mathbb Q$ is generated by the class $\phi$ and the pullbacks of the generators of $\Pic(\cM)\otimes \mathbb Q$. Classically, the standard set of generators of $\Pic(\cM)\otimes \mathbb Q$ consists of the Hodge class $\lambda$ and the classes of boundary divisors $\delta_0,\dots, \delta_{[g/2]}$ (see~\cite{ARBARELLO1987153}). We will keep the same notation for the pullbacks of these classes to $P\cMhit$. Let $\delta = \sum_{j = 0}^{[g/2]}\delta_j$ denote the full boundary class. Let $\nu: \cMpt\to \cM$ be the universal curve and let $\omega_{\nu}$ denote the relative dualizing sheaf. Pulling the Mumford's formula for $\nu_*c_1(\omega_{\nu})^2$ to $P\cMhit$ we get
\begin{equation}
  \pi_*\psi^2 = 12\lambda - \delta
  \label{eq:Mumford_for_cMhit}
\end{equation}
where $\pi$ and $\psi$ where defined in Section~\ref{sec:space_of_covers}.

\subsection{Expansion of the classes of components of the universal discriminant locus}
\label{subsec:formulas_for_three_divisors}

In this section we will prove Theorem~\ref{thma:formulas_for_three_divisors}. The following lemma is straightforward:
\begin{lemma}
  Let $X$ be a complex orbifold, let $h: L\to X$ be a line bundle and $n$ be an integer. Consider the factor space
  \begin{equation*}
    L_{(n)} = \{(v,\alpha)\in L\times \mathbb C\}/_{\sim}
  \end{equation*}
   modulo the relation $(\xi v, \alpha)\sim (v,\xi^n \alpha)$ that holds for any $\xi\in \mathbb C^*$. Then $L_{(n)}$ is a complex orbifold, the projection $L_{(n)}\to X$ given by $(v,\alpha)\mapsto h(v)$ is a line bundle on $X$ and the map $(v,\alpha)\mapsto \alpha\cdot v^{\otimes n}$ is an isomorphism between $L_{(n)}$ and $L^{\otimes n}$.
  \label{lemma:construction_of_Fn}
\end{lemma}

As above we denote by $\hpi: \hC\to P\cMhit$ the universal Hitchin's spectral curve and by $\pi:\cC\to P\cMhit$ is the universal curve over $P\cMhit$. The branched cover $p: \hC\to \cC$ is given fiber-wise by the projection $\hSigma\to\Sigma$. The divisor $\hB\subset \hC$ denotes the ramification divisor of $p$ and the divisor $\B = p(\hB)\subset \cC$ is the branching divisor of $p$. The class $c_1(\omega_{\pi})$ is denoted by $\psi$ as before.
\begin{lemma}
  The following relation holds in $\Pic(P\cMhit)\otimes \mathbb Q$:
  \begin{equation*}
    \hpi_* (p^*\psi\cdot \Bigl[\hB \Bigr]) = n(n-1)(12\lambda - \delta - 2(g-1)\phi).
  \end{equation*}
  \label{lemma:psi_1_cdot_B}
\end{lemma}
\begin{proof}
  It follows from Lemma~\ref{lemma:hB_is_smooth} and the construction of $\hD$ that the projection $\hB\to \B$ is of degree one. Therefore,
  \begin{equation}
    \hpi_*\Bigl[ p^*\psi\cdot \hB \Bigr] = \pi_*p_*\Bigl[ p^*\psi\cdot \hB \Bigr] = \pi_*\Bigl[ \psi\cdot p_*\hB \Bigr] = \pi_*\Bigl[ \psi\cdot \B \Bigr]
    \label{eq:psi_1_cdot_B_part_1}
  \end{equation}
  Define the map $h: \pi^*\L_{(n(n-1))}\to \omega_{\pi}^{\otimes n(n-1)}$ (cf. Lemma~\ref{lemma:construction_of_Fn}) as follows. Let $(\Sigma, x, P)\in \cC$, so that $P\in \L|_{(\Sigma, [P])}$. Set $h(P,\xi) = \xi\cdot \Disc(P)|_x$. Lemma~\ref{lemma:construction_of_Fn} implies that $h(P,\xi)$ is well-defined and depends linearly of $(P, \xi)\in \L_{(n(n-1))}$. Moreover, we have 
  \begin{equation*}
    \div\, h = \B.
  \end{equation*}
  Therefore
  \begin{equation}
    \pi_*\Bigl[ \psi\cdot \B \Bigr] = \pi_*\Bigl[ \psi\cdot (n(n-1)(\psi-\pi^*\phi)) \Bigr] = n(n-1)(12\lambda - \delta - 2(g-1)\phi)
    \label{eq:psi_1_cdot_B_part_2}
  \end{equation}
 where we used~\eqref{eq:Mumford_for_cMhit} in the last equation. Combining~\eqref{eq:psi_1_cdot_B_part_1} and~\eqref{eq:psi_1_cdot_B_part_2} we get the result.
\end{proof}

\begin{proof}[Proof of Theorem~\ref{thma:formulas_for_three_divisors}]
  Let $(\Sigma, x, P, v)$ be a point in $\hB$, so that $(x,v)\in \hSigma$ is a ramification point of $\hSigma\to \Sigma$. Let $P\hDW\subset \hB$ denote the closure of the locus in $\hB$ that parametrizes those $(\Sigma, x, P, v)$ for which $x$ is a double zero of $\Disc(P)$. Then $\hpi(P\hDW) = \supp(P\cDW)$ (for a divisor $D = a_1D_1 + \dots + a_kD_k$ we denote the support by $\supp(D) = D_1\cup\dots\cup D_k$). The divisor $P\hDW$ splits into three components $P\hnod$, $P\hMax$ and $P\hcau$ in accordance with the three possibilities described in the beginning of Section~\ref{sec:discriminant_locus}. We have $\hpi_*P\hnod = P\cnod$, $\hpi_*P\hMax = 2P\cMax$ and $\hpi_*P\hcau = P\ccau$.
  
  Now let $(\Sigma, x_0, P, v_0)$ be a generic point in $\hB$. Without loss of generality we can assume that $\Sigma$ is smooth at $x_0$. Let $U\subset \Sigma$ be a small neighborhood of $x_0$ and $v$ be a holomorphic 1-differential on $U$ such that $v(x_0) = v_0$. Consider the polynomial
  \begin{equation*}
    P(t+v(x), x) = t^n + q_1(x)t^{n-1} + \dots + q_n(x)
  \end{equation*}
   where $x\in U$. Note that we have an equality $\Disc(P(t+v(x),x)) = \Disc(P(t, x))$ for discriminates with respect to $t$, because the discriminant is invariant under an argument shift. Put
  \begin{equation*}
    P_{n-2}(t,x) = t^{n-2} + q_1(x) t^{n-3} + \dots + q_{n-2}(x)
  \end{equation*}
  as in Lemma~\ref{lemma:decomposition_of_Disc}. Let $z$ be a local coordinate on $\Sigma$ at $x_0$ such that $z(x_0)=0$. Since $t=0$ is a zero of the second order of $P(t + v_0, x_0)$ we have $q_{n-1}(x) = O(z(x))$ and $q_n(x) = O(z(x))$. Using the decomposition of the discriminant obtained in Lemma~\ref{lemma:decomposition_of_Disc} we get
  \begin{equation}
    \Disc(P(t+v(x),x)) = -q_n(x)(q_{n-2}(x))^3\,\Disc(P_{n-2}(t,x)) + O(z(x)^2)
    \label{eq:asymptotics_of_Disc}
  \end{equation}
   as $x\to x_0$. Note that the first summand on the right-hand side of~\eqref{eq:asymptotics_of_Disc} has a zero of order $1$ at $x_0$ if the point $(\Sigma, x_0, P, v_0)$ belongs to an non-empty open subset of $\hB$. On the other hand, if the point $(\Sigma, x_0, P, v_0)$ belongs to $P\hDW$, then we have the relation $\Disc(P(t+v(x),x)) = O(z(x)^2)$ as $x\to x_0$, which is equivalent to the condition $q_n(x)(q_{n-2}(x))^3\,\Disc(P_{n-2}(t,x)) = O(z(x)^2)$ on the first summand. Since $q_n(x)(q_{n-2}(x))^3\,\Disc(P_{n-2}(t,x))$ is a product of three differentials we get the following three possibilities for this condition to hold:

  \textbf{1) The formula for $[P\cnod]$.} Assume that $q_n(x) = O(z(x)^2)$ as $x\to x_0$. Then $\Disc(P(t+v(x),x)) = O(z(x)^2)$ as $x\to x_0$ due to~\eqref{eq:asymptotics_of_Disc}. Notice that in this case $v_0$ is a root of order $2$ of $P(t, x_0)$ and $\Disc(P_{n-2}(0,x_0))\neq 0$ in general, thus $(\Sigma, x_0, P, v_0)$ does not belong to $P\hMax$ or $P\hcau$. Therefore, $(\Sigma, x_0, P, v_0)\in P\hnod$. Vice versa, assume that $(\Sigma, x_0, P, v_0)\in P\hnod$. Write $P(v,x) = F(v,x)\,dz^n$, where $F$ is a holomorphic function defined in a neighborhood $T^*U$ of $(x_0,v_0)\in T^*\Sigma$. By the definition $\hSigma\cap T^*U$ is given by the equation $F=0$ in $T^*U$. Since $\hSigma$ is not smooth at $(x_0,v_0)$ we must have $dF(v_0,x_0) = 0$. Because $v_0$ is a root of second order of $P$, we have $dF(v_0, x_0) = \partial_2F(v_0,x_0)$, where $\partial_2$ denotes the partial derivative with respect to the second argument. Now assume that $q_n(x) = f_n(x)\,dz^n$. Then
  \begin{equation}
    dF(v_0, x_0) = \partial_2 F(v_0,x_0) = df_n(x_0).
    \label{eq:partial_derivative_at_node}
  \end{equation}
  Therefore, the equality $dF(v_0,x_0) = 0$ is equivalent to the fact that $q_n(x) = O(z(x)^2)$ as $x\to x_0$. We conclude that the equality $q_n(x) = O(z(x)^2)$ is equivalent to the fact that $(\Sigma, x_0, P, v_0)\in P\hnod$. Introduce the notation
  \begin{equation}
    \Phi(\Sigma, x_0, P, v_0) = dF(v_0, x_0)\,dz^n(x_0)\in (T^*_{x_0}\Sigma)^{\otimes (n+1)}.
    \label{eq:notation_for_derivative}
  \end{equation}
  Since $F(v_0, x_0) = 0$ this is well-defined (i.e. does not depend on the choice of a local coordinate). Note that if $\xi\in \mathbb C^*$ then $\Phi(\Sigma, x_0, \xi\cdot P, \xi v_0) = \xi^n \Phi(\Sigma, x_0, P, v_0)$, where the action of $\mathbb C^*$ is defined by~\eqref{eq:def_of_action_of_Cstar}. It follows from Lemma~\ref{lemma:construction_of_Fn} that $\Phi$ extends to a homomorphism
  \begin{equation}
    \hPhi: \hpi^*\L_{(n+1)}|_{\hB} \to p^*\omega_{\pi}^{\otimes n}|_{\hB}.
    \label{eq:def_of_homo_nod}
  \end{equation}
  defined by $(P, \xi)\mapsto \xi\Phi(\Sigma, x_0, P, v_0)$. Computations made above shows that the vanishing locus of $\hPhi$ coincides with $P\hnod$. Moreover, it is straightforward that $\div\, \hPhi = P\hnod$ so that we have
  \begin{equation}
    P\hnod \equiv (np^*\psi - (n+1)\hpi^*\phi)\cdot \hB
    \label{eq:Phnod_equiv}
  \end{equation}
  in the Chow ring of $\hC$, where we used that $\L_{(n+1)}\simeq L^{\otimes (n+1)}$. Applying Lemma~\ref{lemma:psi_1_cdot_B} we conclude from~\eqref{eq:Phnod_equiv} that
  \begin{equation*}
    [P\cnod] = [\hpi_*P\hnod] = n(n-1)\Bigl( (n+1)(12\lambda-\delta) - 2(g-1)(2n+1)\phi \Bigr).
  \end{equation*}

  \textbf{2) The formula for $[P\cMax]$.} Assume that the equality $\Disc(P_{n-2}(t,x_0)) = 0$ holds. This is equivalent to the fact that $(\Sigma, x_0, P, v_0)\in P\hMax$ by definition of $\hMax$. Introduce the notation 
  \begin{equation*}
    \Phi(\Sigma, x_0, P, v_0) = \Disc(P_{n-2}(t,x_0))\in (T_{x_0}^*\Sigma)^{\otimes (n-2)(n-3)}.
  \end{equation*}
  Note that $\Phi(\Sigma, x_0, P, v_0)$ does not depend on the choice of the differential $v$ (although we used $v$ to define $P_{n-2}$) and we have $\Phi(\Sigma, x_0, \xi\cdot P, \xi v_0) = \xi^{(n-2)(n-3)} \Phi(\Sigma, x_0, P, v_0)$. It follows that $\Phi$ can be extended to a homomorphism
  \begin{equation}
    \hPhi: \hpi^*\L_{((n-2)(n-3))}|_{\hB} \to p^*\omega_{\pi}^{\otimes (n-2)(n-3)}|_{\hB}
    \label{eq:def_of_homo_Max}
  \end{equation}
  defined by $(P, \xi)\mapsto \xi\Phi(\Sigma, x_0, P, v_0)$. We have $\div\,\hPhi = P\hMax$. From here we get that
  \begin{equation}
    P\hMax \equiv (n-2)(n-3)\Bigl( p^*\psi - \hpi^*\phi \Bigr)\cdot \hB
    \label{eq:PhMax_equiv}
  \end{equation}
  in the Chow ring of $\hC$, where by Lemma~\ref{lemma:construction_of_Fn} $\L_{( (n-2)(n-3))}\simeq L^{\otimes (n-2)(n-3)}$. Lemma~\ref{lemma:psi_1_cdot_B} together with the eq.~\eqref{eq:PhMax_equiv} imply that
  \begin{equation*}
    2[P\cMax] = [\hpi_*P\hMax] = n(n-1)(n-2)(n-3)\Bigl( 12\lambda-\delta + 4(g-1)\phi \Bigr).
  \end{equation*}

  \textbf{3) The formula for $[P\ccau]$.} Finally, we consider the case $q_{n-2}(x_0) = 0$ that is equivalent to the fact that $(\Sigma, x_0, P, v_0)\in P\hcau$. Set 
  \begin{equation*}
    \Phi(\Sigma, x_0, P, v_0) = q_{n-2}(x_0)\in (T_{x_0}^*\Sigma)^{\otimes (n-2)}.
  \end{equation*}
  The value $q_{n-2}(x_0)$ is independent of the choice of $v$ and we have $\Phi(\Sigma, x_0, \xi\cdot P, \xi v_0) = \xi^{n-2} \Phi(\Sigma, x_0, P, v_0)$. Hence $\Phi$ can be extended to a homomorphism
  \begin{equation}
    \hPhi: \hpi^*\L_{(n-2)}|_{\hB} \to p^*\omega_{\pi}^{\otimes (n-2)}|_{\hB}
    \label{eq:def_of_homo_cau}
  \end{equation}
  defined by $(P, \xi)\mapsto \xi\Phi(\Sigma, x_0, P, v_0)$ and we have $\div\,\hPhi = P\hcau$. It follows that
  \begin{equation}
    P\hcau \equiv (n-2)\Bigl( p^*\psi - \hpi^*\phi \Bigr)\cdot \hB
    \label{eq:Phcau_equiv}
  \end{equation}
  in the Chow ring of $\hC$, where we use that $\L_{(n-2)}\simeq L^{\otimes (n-2)}$ by Lemma~\ref{lemma:construction_of_Fn}. Lemma~\ref{lemma:psi_1_cdot_B} and~\eqref{eq:Phcau_equiv} imply that
  \begin{equation*}
    [P\ccau] = [\hpi_*P\hcau] = n(n-1)(n-2)\Bigl( 12\lambda-\delta + 4(g-1)\phi \Bigr).
  \end{equation*}
\end{proof}

\section{The Hodge classes on $P\cMhit$.}
\label{sec:hlambda}

In this section we will prove Theorem~\ref{thma:hlambda_formula}. We proceed using the notation introduced in two previous sections.

\begin{lemma}
  The following formula holds in $\Pic(P\cMhit)\otimes\mathbb Q$:
  \begin{equation*}
    [\hpi_* (\hB\cdot \hB)] =  -\frac{n(n-1)}{2}(12\lambda - \delta - 2(g-1)\phi) + \frac{1}{2}([P\cnod] + [P\ccau]).
  \end{equation*}
  \label{lemma:pushforward_of_hB_cdot_hB}
\end{lemma}

\begin{proof}
  Recall that $\hB$ is smooth due to Lemma~\ref{lemma:hB_is_smooth}. The projection $\hB\to P\cMhit$ is a branched cover of degree $2n(n-1)(g-1)$ (equal to the number of zeros of $\Disc(P)$ counted with multiplicities). From the discussion in Section~\ref{sec:discriminant_locus} it follows that the ramification divisor of this branched cover is $P\hnod + P\hcau$. These observations yield the following expression for the canonical class of $\hB$:
  \begin{equation}
    c_1(K_{\hB}) = \hpi^*c_1(K_{P\cMhit})\cdot \hB + P\hnod + P\hcau.
    \label{eq:KhB_first}
  \end{equation}
  Another expression for the canonical class of $\hB$ comes from the adjunction formula:
  \begin{equation}
    c_1(K_{\hB}) = (c_1(K_{\hC}) + \hB)\cdot \hB.
    \label{eq:KhB_second}
  \end{equation}
  Using these two expressions we get
  \begin{equation}
    \hB\cdot\hB = ( \hpi^*c_1(K_{P\cMhit}) - c_1(K_{\hC}) )\cdot \hB + P\hnod + P\hcau = - c_1(\omega_{\hpi})\cdot \hB + P\hnod + P\hcau
    \label{eq:hB_cdot_hB_via_conclasses}
  \end{equation}
  where $\omega_{\hpi}$ is the relative dualizing sheaf. Recall that the map $p: \hC\to \cC$ is a branched cover with a simple ramification along $\hB$ due to Lemma~\ref{lemma:hB_is_smooth}. Therefore,
  \begin{equation}
    c_1(\omega_{\hpi}) = p^*\psi + \hB.
    \label{eq:omega_hpi_is_pullback}
  \end{equation}
  Substituting this expression for $c_1(\omega_{\hpi})$ into~\eqref{eq:hB_cdot_hB_via_conclasses} we find that
  \begin{equation}
    2\hB\cdot\hB = - p^*\psi\cdot \hB + P\hnod + P\hcau.
    \label{eq:2hB_cdot_hB}
  \end{equation}
  Applying $\hpi_*$ to this equation and using Lemma~\ref{lemma:psi_1_cdot_B} we get the statement of the lemma.
\end{proof}

\begin{lemma}
  The following formula holds in $\Pic(P\cMhit)\otimes\mathbb Q$:
  \begin{equation}
    \hpi_*c_1(\omega_{\hpi})^2 = 6n(3n-1)\lambda - 3n(n-1)(g-1)\phi - \frac{n(3n-1)}{2}\,\delta + \frac{1}{2}\left([P\cnod] + [P\ccau]\right).
    \label{eq:pushforward_of_omega_hpi2}
  \end{equation}
  \label{lemma:pushforward_of_omega_hpi2}
\end{lemma}

\begin{proof}
  Using~\eqref{eq:omega_hpi_is_pullback} we can write
  \begin{equation*}
    \hpi_*c_1(\omega_{\hpi})^2 = \hpi_* \Bigl(p^*\psi^2 + 2p^*\psi\cdot \hB + \hB\cdot \hB\Bigr) = n\pi_*\psi^2 +  \hpi_* \Bigl(2p^*\psi\cdot \hB + \hB\cdot \hB\Bigr).
  \end{equation*}
  Combining this formula with the eq.~\eqref{eq:Mumford_for_cMhit}, Lemma~\ref{lemma:psi_1_cdot_B} and Lemma~\ref{lemma:pushforward_of_hB_cdot_hB} we get the desired eq.~\eqref{eq:pushforward_of_omega_hpi2}.
\end{proof}

\begin{proof}[Proof of Theorem~\ref{thma:hlambda_formula}]
  Let $\Vn\subset \hC$ denote the locus of nodal points of fibers of $\hpi$, i.e.
  \begin{equation}
    \Vn = \{(\Sigma, x, P, v)\ \mid\ (x,v)\text{ is a node of $\hSigma$}\}.
    \label{eq:def_of_Vn}
  \end{equation}
  Note that $P\hnod$ is a component of $\Vn$, and we have
  \begin{equation}
    \hpi_*\Vn = n\delta + [P\cnod].
    \label{eq:pushforward_of_Vn}
  \end{equation}
  Now let us apply the Grothendieck-Riemann-Roch formula to the structure sheaf $\O_{\hC}$ and the morphism $\hpi: \hC\to P\cMhit$. Taking the degree one components of both sides of the formula we get the following relation in $\Pic(P\cMhit)\otimes\mathbb Q$:
  \begin{equation}
    12\hlambda = \hpi_* (c_1(\omega_{\hpi})^2 + \Vn).
    \label{eq:GRR_for_structure_sheaf}
  \end{equation}
  Using Lemma~\ref{lemma:pushforward_of_omega_hpi2} and~\eqref{eq:pushforward_of_Vn} we conclude that
  \begin{equation}
    12\hlambda = 6n(3n-1)\lambda - 3n(n-1)(g-1)\phi - \frac{n(3n-3)}{2}\,\delta + \frac{3}{2}[P\cnod] + \frac{1}{2}[P\ccau]
    \label{eq:GRR_intermediate}
  \end{equation}
  The formulas of Theorem~\ref{thma:formulas_for_three_divisors}
  \begin{align*}
    & [P\cnod] = n(n-1)\Bigl( (n+1)(12\lambda-\delta) - 2(g-1)(2n+1)\phi \Bigr)\\
    & [P\ccau] = n(n-1)(n-2)\Bigl( 12\lambda-\delta - 4(g-1)\phi \Bigr).
  \end{align*}
  together with the eq.~\eqref{eq:GRR_intermediate} give the desired formula for $\hlambda$.
\end{proof}

\bibliographystyle{habbrv} 
\bibliography{smth}

\end{document}